\documentclass[11pt,a4paper]{article}
\usepackage{amssymb}
\usepackage{amsmath}
\usepackage{harvard}
\usepackage{cite}

\newtheorem{theorem}{Theorem}[section]

\newtheorem{corollary}[theorem]{Corollary}

\newtheorem{definition}[theorem]{Definition}
\newtheorem{example}[theorem]{Example}

\newtheorem{lemma}[theorem]{Lemma}

\newtheorem{remark}[theorem]{Remark}

\newenvironment{proof}[1][Proof]{\noindent\textbf{#1.} }{\ \rule{0.5em}{0.5em}}
\input{tcilatex}
\begin{document}

\title{Topology of certain symplectic conifold transitions of $\mathbb{C}%
P^{1}-$bundles}
\author{Yi Jiang \\
%EndAName
Mathematical Sciences Center, Tsinghua University, \\
Beijing 100084, China}
\maketitle

\begin{abstract}
In this paper, we prove the existence of certain symplectic conifold
transitions on all $\mathbb{C}P^{1}-$bundles over symplectic 4--manifolds,
which generalizes Smith, Thomas and Yau's examples of symplectic conifold
transitions on trivial $\mathbb{C}P^{1}-$bundles over K\"{a}hler surfaces.
Our main result is to determine the diffeomorphism types of such symplectic
conifold transitions of $\mathbb{C}P^{1}-$bundles. In particular, this
implies that in the case of trivial $\mathbb{C}P^{1}-$bundles over
projective surfaces, Smith, Thomas and Yau's examples of symplectic conifold
transitions are diffeomorphic to K\"{a}hler 3--folds.

\begin{description}
\item[2010 Mathematics subject classification:] 57R17 (57R22)

\item[Key Words and phrases:] symplectic conifold transitions, $\mathbb{C}%
P^{1}-$bundles, K\"{a}hler 3--folds, characteristic classes

\item[Email address:] yjiang117@mail.tsinghua.edu.cn
\end{description}
\end{abstract}

\section{Introduction}

In this paper, all manifolds under consideration are closed, oriented and
differentiable, unless otherwise stated. By a $\mathbb{C}P^{1}-$bundle, we
always mean the projectivization $\mathbb{P}\left( E\right) $ of a complex
vector bundle $E$ of rank two.

Symplectic conifold transitions introduced by Smith, Thomas and Yau \cite%
{STY} are symplectic surgeries on symplectic 6--manifolds which collapses
embedded Lagrangian 3--spheres and replaces them by symplectic 2--spheres.
One sufficient condition to realize such a symplectic surgery is the
existence of a nullhomology Lagrangian 3--sphere in the initial symplectic
6--manifold \cite[Theorem 2.9]{STY}. As a family of typical symplectic
6--manifolds constructed by Thurston \cite[ Theorem 6.3]{McS}, $\mathbb{C}%
P^{1}-$bundles over symplectic 4--manifolds can be considered as the initial
6--manifolds and one can study the existence of nullhomology Lagrangian
3--spheres in them. In trivial $\mathbb{C}P^{1}-$bundles over K\"{a}hler
surfaces, Smith, Thomas and Yau \cite{STY} found some nullhomology
Lagrangian 3--spheres and gave examples of symplectic conifold transitions
along these Lagrangian 3--spheres which will be called \textsl{canonical} in
our paper (see Definition \ref{def}). They pointed out that these examples
can produce 3--folds with arbitrarily high second Betti number which are not
obviously blowups of smooth 3--folds and it should be possible for them to
contain non--K\"{a}hler examples. Indeed, Corti and Smith \cite{CS} proved
that there is such a symplectic conifold transition of the trivial $\mathbb{C%
}P^{1}-$bundle over some Enriques surface which is not deformation
equivalent to any K\"{a}hler 3--fold.

However, our main results in this paper will imply that Smith, Thomas and
Yau's examples of symplectic conifold transitions of trivial $\mathbb{C}%
P^{1}-$bundles are diffeomorphic to either $\mathbb{C}P^{1}-$bundles or
blowups of $\mathbb{C}P^{1}-$bundles; in particular, Corti and Smith's
examples of symplectic conifold transitions are diffeomorphic to K\"{a}hler
3--folds. More generally, we find canonical Lagrangian 3--spheres in all $%
\mathbb{C}P^{1}-$bundles over symplectic 4--manifolds (see Lemma \ref{ml1})
and prove Theorem \ref{mt}, Corollary \ref{mc} below.

For simplicity, denote $\overline{%
%TCIMACRO{\U{2102} }%
%BeginExpansion
\mathbb{C}
%EndExpansion
P^{2}}$ and $S^{4}$ by $N_{k}$, $k=1,2$, respectively, where $\overline{%
%TCIMACRO{\U{2102} }%
%BeginExpansion
\mathbb{C}
%EndExpansion
P^{n}}$ denotes the complex projective space $%
%TCIMACRO{\U{2102} }%
%BeginExpansion
\mathbb{C}
%EndExpansion
P^{n}$ with the opposite orientation. For $k=1,2$, let $\sigma _{k}\in
H_{2}(N_{k})$ and $\sigma _{k}^{\ast }\in H^{2}(N_{k})$ such that $\sigma
_{1}^{\ast }$ is the dual class of the preferred generator $\sigma _{1}$,
i.e. $\left\langle \sigma _{1}^{\ast },\sigma _{1}\right\rangle =1$, and $%
\sigma _{2}=0$, $\sigma _{2}^{\ast }=0$. Denote $\left[ M\right] $ for the
fundamental class of a manifold $M$. As there are exactly two distinct
conifold transitions along a Lagrangian 3--sphere up to diffeomorphism\cite%
{STY}, we can state our main results as following.

\begin{theorem}
\label{mt}Let $\mathbb{P}\left( E\right) $ be a symplectic manifold which is
the projectivization of a rank two complex vector bundle $E$ over a
4--manifold $N$. Suppose $\mathbb{P}\left( E\right) $ has a canonical
Lagrangian 3--sphere. Then the two symplectic conifold transitions of $%
\mathbb{P}\left( E\right) $ along this Lagrangian 3--sphere are
diffeomorphic to $\mathbb{P}(E_{1})$ and the connected sum $\mathbb{P}%
(E_{2})\sharp \overline{%
%TCIMACRO{\U{2102} }%
%BeginExpansion
\mathbb{C}
%EndExpansion
P^{3}}$ respectively, where $E_{k}$, $k=1,2$ are the rank two complex
bundles over $N\sharp N_{k}$ with Chern classes satisfying%
\begin{eqnarray*}
c_{1}(E_{k}) &=&(c_{1}(E),-\sigma _{k}^{\ast })\in H^{2}(N\sharp N_{k})\cong
H^{2}(N)\oplus H^{2}(N_{k}); \\
\left\langle c_{2}(E_{k}),\left[ N\sharp N_{k}\right] \right\rangle 
&=&\left\langle c_{2}(E),\left[ N\right] \right\rangle -1.
\end{eqnarray*}%
Moreover, if $N$ is symplectic, then the above diffeomorphisms can be chosen
to preserve the first Chern classes.
\end{theorem}

\begin{remark}
For a 4--manifold $N$, every pair in $H^{2}(N)\times H^{4}(N)$ can be
realized as the Chern classes of a rank two complex vector bundles $E$ over $%
N$ and the isomorphism classes of the bundles $E_{k}$ in Theorem \ref{mt}
can be completely determined by $c_{1}(E_{k}),c_{2}(E_{k})$\cite[Theorem
1.4.20]{GS}. Moreover, it is not hard to prove the manifolds $\mathbb{P}%
(E_{1})$ and $\mathbb{P}(E_{2})\sharp \overline{%
%TCIMACRO{\U{2102} }%
%BeginExpansion
\mathbb{C}
%EndExpansion
P^{3}}$ are in different diffeomorphism classes by comparing the cohomology
rings.

The existence of diffeomorphisms preserving the first Chern classes $c_{1}$
can be used to define an\ equivalence relation between symplectic
6--manifolds\cite[2.1(D)]{Sa}; for almost complex 6--manifolds, a
diffeomorphism preserving $c_{1}$ means it preserves the almost complex
structures\cite[Theorem 9]{W}.
\end{remark}

As it is well--known that the projectivization of a holomorphic vector
bundle over a K\"{a}hler manifold admits a K\"{a}hler structure\cite[%
Proposition 3.18]{V}, we can obtain the following corollary from Theorem \ref%
{mt}.

\begin{corollary}
\label{mc}Let $\mathbb{P}\left( E\right) $ be the projectivization of a rank
two holomorphic vector bundle $E$ over a projective surface $N$, then the
symplectic conifold transitions of $\mathbb{P}\left( E\right) $ along a
canonical Lagrangian 3--sphere are diffeomorphic to K\"{a}hler 3--folds.
\end{corollary}

We will finish the proof of Theorem \ref{mt} and Corollary \ref{mc} in
Section 3.3. In the course of establishing Theorem \ref{mt}, we also compute
the topological invariants of $\mathbb{C}P^{1}-$bundles over
simply--connected 4--manifolds in Example \ref{ExP}. According to \cite{W}
and \cite{V}, combining this computation with Theorem \ref{mt} will give
diffeomorphism classification of symplectic conifold transitions of
simply--connected $\mathbb{C}P^{1}-$bundles along canonical Lagrangian
3--spheres.

\section{Symplectic conifold transitions on $\mathbb{C}P^{1}-$bundles}

We first recall the definition of conifold transitions \cite{STY}. Begin
with a Lagrangian embedding $f:S^{3}\rightarrow X$ in a symplectic
6--manifold $X$. By the Lagrangian neighborhood theorem \cite[Theorem 3.33]%
{McS}, the embedding $f$ can extend to a symplectic embedding $f^{\prime
}:T_{\epsilon }^{\ast }S^{3}\rightarrow X$ with $T_{\epsilon }^{\ast
}S^{3}\subset T^{\ast }S^{3}$ a neighborhood of the zero section of the
cotangent bundle. Define a \textit{conifold transition} along $f$ to be the
smooth manifold%
\begin{equation*}
Y_{k}:=X\backslash f[S^{3}]\cup _{f^{\prime }\circ \psi _{k}}W_{k}^{\epsilon
}
\end{equation*}%
for $k=1,2$, where $W_{k}$ are two small resolutions of the complex
singularity $W=\left\{ \sum z_{j}^{2}=0\right\} \subset 
%TCIMACRO{\U{2102} }%
%BeginExpansion
\mathbb{C}
%EndExpansion
^{4}$ with exceptional set $\mathbb{C}P^{1}$ over $\left\{ 0\right\} \in W$
and either of $W_{k}$ is a complex vector bundle over $\mathbb{C}P^{1}$ with
first Chern number $-2$; fixing coordinates on $T^{\ast }S^{3}$ as%
\begin{equation*}
T^{\ast }S^{3}=\left\{ \left( u,v\right) \in 
%TCIMACRO{\U{211d} }%
%BeginExpansion
\mathbb{R}
%EndExpansion
^{4}\times 
%TCIMACRO{\U{211d} }%
%BeginExpansion
\mathbb{R}
%EndExpansion
^{4}|\left\vert u\right\vert =1,\left\langle u,v\right\rangle =0\right\} 
\text{,}
\end{equation*}%
the maps $\psi _{k}:(W_{k}\backslash \mathbb{C}P^{1}\cong W\backslash
\{0\},\omega _{%
%TCIMACRO{\U{2102} }%
%BeginExpansion
\mathbb{C}
%EndExpansion
})\rightarrow (T^{\ast }S^{3}\backslash \{v=0\},d(vdu))$ are
symplectomorphisms defined in \cite[(2.1)]{STY} with $\omega _{%
%TCIMACRO{\U{2102} }%
%BeginExpansion
\mathbb{C}
%EndExpansion
}$ the restriction of the symplectic form $\frac{i}{2}\sum_{j}dz_{j}\wedge d%
\overline{z_{j}}$ on $%
%TCIMACRO{\U{2102} }%
%BeginExpansion
\mathbb{C}
%EndExpansion
^{4}$; the submanifolds $W_{k}^{\epsilon }\subset W_{k}$ are neighborhoods
of the exceptional set $\mathbb{C}P^{1}$ such that $W_{k}^{\epsilon
}\backslash \mathbb{C}P^{1}=\psi _{k}^{-1}[T_{\epsilon }^{\ast
}S^{3}\backslash \{v=0\}]$.

There are more choices in conifold transitions along a Lagrangian 3--sphere
than along a Lagrangian embedding $S^{3}\rightarrow X$, as changing the
orientation of the Lagrangian 3--sphere $f[S^{3}]$ would induce a new
Lagrangian embedding $S^{3}\rightarrow X$ different from $f$. However, this
change would just swap the diffeomorphism types of the conifold transitions,
so there are exactly two distinct conifold transitions $Y_{k}$, $k=1,2$
along the Lagrangian 3--sphere $f[S^{3}]$ up to diffeomorphism\cite{STY}. It
follows from \cite[Theorem 2.9]{STY} that the two conifold transitions of a
symplectic 6--manifold along a nullhomology Lagrangian 3--sphere both admit
distinguished symplectic structures. Hence to realize such symplectic
conifold transitions on $\mathbb{C}P^{1}-$bundles, it suffices to find
nullhomology Lagrangian 3--spheres.

Inside the product $(%
%TCIMACRO{\U{2102} }%
%BeginExpansion
\mathbb{C}
%EndExpansion
^{2}\times \mathbb{C}P^{1},\omega _{0}\times \omega _{\mathbb{C}P^{1}})$ of
symplectic manifolds with $\omega _{0}=\frac{i}{2}\sum_{j}dz_{j}\wedge d%
\overline{z_{j}}$ on $%
%TCIMACRO{\U{2102} }%
%BeginExpansion
\mathbb{C}
%EndExpansion
^{2}$ and $\omega _{\mathbb{C}P^{1}}$ the Fubini--Study form on $\mathbb{C}%
P^{1}$, a well--known construction\cite{ALP} of a nullhomology Lagrangian
3--sphere is given by the composition of maps%
\begin{equation}
f:S^{3}\overset{(i,h)}{\rightarrow }%
%TCIMACRO{\U{2102} }%
%BeginExpansion
\mathbb{C}
%EndExpansion
^{2}\times \mathbb{C}P^{1}\overset{\iota \times id_{\mathbb{C}P^{1}}}{%
\rightarrow }%
%TCIMACRO{\U{2102} }%
%BeginExpansion
\mathbb{C}
%EndExpansion
^{2}\times \mathbb{C}P^{1}  \tag{1.1}  \label{mapf}
\end{equation}%
where $i:S^{3}\subset 
%TCIMACRO{\U{2102} }%
%BeginExpansion
\mathbb{C}
%EndExpansion
^{2}$ is the inclusion of the unit sphere, $h:S^{3}\rightarrow \mathbb{C}%
P^{1}$ is the Hopf map and $\iota $ is the complex conjugation on $%
%TCIMACRO{\U{2102} }%
%BeginExpansion
\mathbb{C}
%EndExpansion
^{2}$. As the image $f[S^{3}]$ entirely contains in $B^{4}(l)\times \mathbb{C%
}P^{1}$ with $B^{4}(l)$ a ball in $%
%TCIMACRO{\U{2102} }%
%BeginExpansion
\mathbb{C}
%EndExpansion
^{2}$ of radius $l>1$, hence finding symplectic embeddings of $%
B^{4}(l)\times \mathbb{C}P^{1}$ in $\mathbb{C}P^{1}-$bundles would give
nullhomology Lagrangian 3--spheres in the bundles. This may lead to the
following definition.

\begin{definition}
\label{def} Let $\mathbb{P}\left( E\right) $ be a symplectic manifold which
is a $\mathbb{C}P^{1}-$bundle over a 4--manifold $N$. A Lagrangian 3--sphere
in $\mathbb{P}\left( E\right) $ is called \textsl{canonical} if it is the
image of the composition of embeddings%
\begin{equation*}
S^{3}\overset{f}{\rightarrow }B^{4}(l)\times \mathbb{C}P^{1}\overset{\eta }{%
\rightarrow }\mathbb{P}\left( E\right)
\end{equation*}%
for $l>1$ where the symplectic embedding $\eta $ can induce a local
trivialization of the bundle $\pi :\mathbb{P}\left( E\right) \rightarrow N$,
i.e. there is a differentiable embedding $k:B^{4}(l)\rightarrow N$ such that 
\begin{equation*}
\pi ^{-1}[k[B^{4}(l)]]=\eta \lbrack B^{4}(l)\times \mathbb{C}P^{1}]\overset{%
\eta ^{-1}}{\rightarrow }B^{4}(l)\times \mathbb{C}P^{1}\overset{k\times id_{%
\mathbb{C}P^{1}}}{\rightarrow }k[B^{4}(l)]\times \mathbb{C}P^{1}
\end{equation*}%
is a local trivialization of the $\mathbb{C}P^{1}-$bundle $\mathbb{P}\left(
E\right) $.
\end{definition}

\cite{STY} and \cite{CS} have shown the existence of canonical Lagrangian
3--spheres in trivial $\mathbb{C}P^{1}-$bundles over K\"{a}hler surfaces. We
generalize their result in the following Lemma by Thurston's construction 
\cite[Theorem 6.3]{McS} and the construction of K\"{a}hler forms on $\mathbb{%
P}\left( E\right) $\cite[Proposition 3.18]{V}.

\begin{lemma}
\label{ml1}Let $\mathbb{P}\left( E\right) $ be the projectivization of a
rank two complex vector bundle $E$ over a symplectic 4--manifold $N$. Then $%
\mathbb{P}\left( E\right) $ admits a symplectic form such that it has a
embedded canonical Lagrangian 3--sphere. Moreover, if $N$ is K\"{a}hler and $%
E$ admits a holomorphic structure, then $\mathbb{P}\left( E\right) $ admits
a K\"{a}hler form such that it has a embedded canonical Lagrangian
3--sphere.\medskip 

\begin{proof}
It suffices to find a symplectic embedding $\eta :B^{4}(l)\times \mathbb{C}%
P^{1}\rightarrow \mathbb{P}\left( E\right) $ which can induce a local
trivialization with $l>1$. The keypoint is to note that there exists a
system of local trivializations $\left\{ (U_{j},\phi _{j})\right\} _{j=0}^{m}
$ of the bundle $\pi :\mathbb{P}\left( E\right) \rightarrow N$ and a
partition of unity $\rho _{j}:N\rightarrow \lbrack 0,1]$ subordinating to
the open cover $\left\{ U_{j}\right\} _{j=0}^{m}$ of $N$ such that each $%
U_{j}$ is contractible and $\rho _{0}\equiv 1$ on a nonempty open subset $%
V\subset U_{0}$. In fact, this follows easily from \cite[Corollary 5.2]{BT} .

For the symplectic case, apply Thurston's construction of the symplectic
form to $\mathbb{P}\left( E\right) $. Let $L^{\ast }$ denote the dual bundle
of the tautological line bundle $L=\left\{ \left( l,v\right) \in \mathbb{P}%
\left( E\right) \times E\text{ }|\text{ }v\in l\right\} $ over $\mathbb{P}%
\left( E\right) $. According to the proof of \cite[ Theorem 6.3]{McS}, the
first Chern class $c_{1}(L^{\ast })\in H^{2}(\mathbb{P}\left( E\right) )$,
the local trivializations $\left\{ (U_{j},\phi _{j})\right\} _{j=0}^{m}$ and
the partition of unity $\rho _{j}:N\rightarrow \lbrack 0,1]$ can contribute
to define a closed 2--form $\tau $ on $\mathbb{P}\left( E\right) $ such that
the restriction of $\tau $ on each fiber $\mathbb{C}P^{1}$ is just $\omega _{%
\mathbb{C}P^{1}}$. Moreover, since $\rho _{0}\equiv 1$ on $V$, then the form 
$\tau $ can be chosen such that its restriction on $\pi ^{-1}[V]$ is equal
to the pullback $\phi _{0}^{\ast }0\times \omega _{\mathbb{C}P^{1}}$ of the
form $0\times \omega _{\mathbb{C}P^{1}}$ on $V\times \mathbb{C}P^{1}$. By 
\cite[ Theorem 6.3]{McS}, the 2--form $\tau +\lambda \pi ^{\ast }\omega _{N}$
on $\mathbb{P}\left( E\right) $ is symplectic for $\lambda >0$ sufficiently
large where $\omega _{N}$ denotes the symplectic form on $N$. By the Darboux
neighborhood theorem, there is always a symplectic embedding $%
B^{4}(l)\rightarrow (V,\lambda \omega _{N})$ with $l>1$ for $\lambda $
sufficiently large. So in this case, we have the composition of symplectic
embeddings%
\begin{equation*}
B^{4}(l)\times \mathbb{C}P^{1}\rightarrow (V\times \mathbb{C}P^{1},\lambda
\omega _{N}\times \omega _{\mathbb{C}P^{1}})\overset{\phi _{0}^{-1}}{%
\rightarrow }(\mathbb{P}\left( E\right) ,\tau +\lambda \pi ^{\ast }\omega
_{N})
\end{equation*}%
which is the desired embedding.

Now for the K\"{a}hler case, assume $\omega _{N}$ is the K\"{a}hler form on $%
N$ and $E$ is holomorphic. Using the system of local trivializations $%
\left\{ (U_{j},\varphi _{j})\right\} _{j=0}^{m}$ of $E$ associated to $%
\left\{ (U_{j},\phi _{j})\right\} _{j=0}^{m}$ and the partition of unity $%
\rho _{j}$, we can obtain a Hermitian metric $h$ on $E$ such that on the
restriction $E|_{V}$ of $E$ to $V$, the metric $h$ is induced by the
canonical Hermitian metric on $%
%TCIMACRO{\U{2102} }%
%BeginExpansion
\mathbb{C}
%EndExpansion
^{2}$ via the projection $E|_{V}\overset{\varphi _{0}}{\rightarrow }V\times 
%TCIMACRO{\U{2102} }%
%BeginExpansion
\mathbb{C}
%EndExpansion
^{2}\rightarrow 
%TCIMACRO{\U{2102} }%
%BeginExpansion
\mathbb{C}
%EndExpansion
^{2}$. \cite[Proposition 3.18]{V} shows that $h$ induces a Hermitian metric
on the bundle $L^{\ast }$ and the Chern form $\omega _{E}$ associated to
this metric can contribute to obtain a K\"{a}hler form $\omega _{E}+\lambda
\pi ^{\ast }\omega _{N}$ for $\lambda >0$ sufficiently large. Replacing $%
\tau $ by $\omega _{E}$ in proof of the symplectic case, we can get the
desired symplectic embedding. This completes the proof.
\end{proof}
\end{lemma}

\section{Topology of symplectic conifold transitions of $\mathbb{C}P^{1}-$%
bundles}

The aim of this section is to study the topology of symplectic conifold
transitions of $\mathbb{C}P^{1}-$bundles along canonical Lagrangian
3--spheres and prove Theorem \ref{mt} and Corollary \ref{mc}. For this
purpose, we first recall in Section 3.1 the invariants of simply--connected
6--manifolds with torsion--free homology, and compute the invariants of $%
\mathbb{C}P^{1}-$bundles over simply--connected 4--manifolds; then determine
in Section 2 the topology of conifold transitions of $B^{4}(l)\times \mathbb{%
C}P^{1}$ along $f[S^{3}]$, i.e. to establish Lemma \ref{ml}, which is a
keypoint to prove Theorem \ref{mt}.

\subsection{Invariants of simply--connected 6--manifolds with torsion--free
homology}

By Wall\cite{W} and Jupp\cite{V}, the third Betti number $b_{3}$, the
integral cohomology ring $H^{\ast }$, the first Pontrjagin class $p_{1}$ and
the second Whitney--Stiefel class $w_{2}$ form a system of invariants, which
can distinguish all diffeomorphism classes of simply--connected 6--manifolds
with torsion--free homology. As an example, we will compute these invariants
for $\mathbb{C}P^{1}-$bundles over simply--connected 4--manifolds.

\begin{example}
\label{ExP} Let $\pi :\mathbb{P}\left( E\right) \rightarrow N$ be the
projectivization of a rank two complex vector bundle $E$ over a
simply--connected 4--manifold. Then $\mathbb{P}\left( E\right) $ has a
natural orientation which is compatible with that of the base and fibers. By
the homotopy exact sequence and Gysin sequence, the 6--manifold $\mathbb{P}%
\left( E\right) $ is a simply--connected with $b_{3}=0$. The cohomology ring
and the characteristic classes $w_{2}$, $p_{1}$ can be computed as follows. 
\newline
\textbf{(i)}By the definition of Chern classes\cite[Section 20]{BT}, we have 
\begin{equation*}
H^{\ast }(\mathbb{P}\left( E\right) )\cong H^{\ast }(N)[a]/\left\langle
a^{2}+\pi ^{\ast }c_{1}(E)\cdot a+\pi ^{\ast }c_{2}(E)\right\rangle
\end{equation*}%
where $a=c_{1}(L^{\ast })$ with $L^{\ast }$ the dual bundle of the
tautological line bundle $L=\left\{ \left( l,v\right) \in \mathbb{P}\left(
E\right) \times E\text{ }|\text{ }v\in l\right\} $ over $\mathbb{P}\left(
E\right) $. Let $\left\{ y_{i}\right\} $ be a basis of the free $%
%TCIMACRO{\U{2124} }%
%BeginExpansion
\mathbb{Z}
%EndExpansion
-$module $H^{2}(N)$ and then $\left\{ a,\pi ^{\ast }y_{i}\right\} $ forms a
basis of $H^{2}(\mathbb{P}\left( E\right) )$. By the relations $a^{2}+\pi
^{\ast }c_{1}(E)\cdot a+\pi ^{\ast }c_{2}(E)=0$ and $\left\langle \left[ N%
\right] ^{\ast }\cup a,\left[ \mathbb{P}\left( E\right) \right]
\right\rangle =1$ with $\left[ N\right] ^{\ast }\in H^{4}(N)$ satisfying $%
\left\langle \left[ N\right] ^{\ast },\left[ N\right] \right\rangle =1$, we
can obtain 
\begin{eqnarray*}
\left\langle a^{3},\left[ \mathbb{P}\left( E\right) \right] \right\rangle
&=&\left\langle c_{1}(E)^{2}-c_{2}(E),\left[ N\right] \right\rangle ; \\
\left\langle a^{2}\cup \pi ^{\ast }y_{i},\left[ \mathbb{P}\left( E\right) %
\right] \right\rangle &=&-\left\langle c_{1}(E)y_{i},\left[ N\right]
\right\rangle ; \\
\left\langle a\cup \pi ^{\ast }y_{i}\cup \pi ^{\ast }y_{j},\left[ \mathbb{P}%
\left( E\right) \right] \right\rangle &=&\left\langle y_{i}y_{j},\left[ N%
\right] \right\rangle ; \\
\left\langle \pi ^{\ast }y_{i}\cup \pi ^{\ast }y_{j}\cup \pi ^{\ast }y_{k}, 
\left[ \mathbb{P}\left( E\right) \right] \right\rangle &=&0.
\end{eqnarray*}%
\textbf{(ii)}As the tautological line bundle $L$ is a subbundle of the
pullback $\pi ^{\ast }E$ and a Hermitian metric on $E$ pulls back to a
Hermitian metric on $\pi ^{\ast }E$, we have a splitting $\pi ^{\ast
}E=L\oplus L^{\perp }$ where $L^{\perp }$ is the orthogonal complement
bundle of $L$ and hence 
\begin{eqnarray*}
T\mathbb{P}\left( E\right) &\cong &\pi ^{\ast }TN\oplus Hom_{%
%TCIMACRO{\U{2102} }%
%BeginExpansion
\mathbb{C}
%EndExpansion
}(L,L^{\perp })\text{\cite[Theorem 14.10]{MiS}}; \\
Hom_{%
%TCIMACRO{\U{2102} }%
%BeginExpansion
\mathbb{C}
%EndExpansion
}(L,L^{\perp })\oplus \varepsilon _{%
%TCIMACRO{\U{2102} }%
%BeginExpansion
\mathbb{C}
%EndExpansion
}^{1} &\cong &L^{\ast }\otimes \pi ^{\ast }E
\end{eqnarray*}%
with $\varepsilon _{%
%TCIMACRO{\U{2102} }%
%BeginExpansion
\mathbb{C}
%EndExpansion
}^{1}$ the trivial complex line bundle. These isomorphisms, together with
the relations $a^{2}+\pi ^{\ast }c_{1}(E)\cdot a+\pi ^{\ast }c_{2}(E)=0$, $%
p_{1}=c_{1}{}^{2}-2c_{2}$ and%
\begin{equation*}
c_{1}(L_{1}\otimes L_{2})=2c_{1}(L_{1})+c_{1}(L_{2});c_{2}(L_{1}\otimes
L_{2})=c_{1}(L_{1})^{2}+c_{1}(L_{1})c_{1}(L_{2})+c_{2}(L_{2})
\end{equation*}%
with $L_{i}$ a complex vector bundle of rank $i=1,2$\cite[Problem 16-B]{MiS}%
, imply%
\begin{eqnarray*}
w_{2}(T\mathbb{P}\left( E\right) ) &\equiv &\pi ^{\ast }(w_{2}(TN)+w_{2}(E));
\\
p_{1}(T\mathbb{P}\left( E\right) ) &=&\pi ^{\ast
}(p_{1}(TN)+c_{1}(E)^{2}-4c_{2}(E)).
\end{eqnarray*}%
Thus we have%
\begin{eqnarray*}
\left\langle a^{2}\cup w_{2}(T\mathbb{P}\left( E\right) ),\left[ \mathbb{P}%
\left( E\right) \right] \right\rangle &=&\left\langle w_{2}(E)\cup
(w_{2}(E)+w_{2}(TN)),\left[ \mathbb{P}\left( E\right) \right] \right\rangle ;
\\
\left\langle a\cup \pi ^{\ast }y_{i}\cup w_{2}(T\mathbb{P}\left( E\right) ),%
\left[ \mathbb{P}\left( E\right) \right] \right\rangle &=&\left\langle
y_{i}\cup (w_{2}(E)+w_{2}(TN)),\left[ \mathbb{P}\left( E\right) \right]
\right\rangle ; \\
\left\langle \pi ^{\ast }y_{i}\cup \pi ^{\ast }y_{j}\cup w_{2}(T\mathbb{P}%
\left( E\right) ),\left[ \mathbb{P}\left( E\right) \right] \right\rangle
&=&0. \\
\left\langle a\cup p_{1}(T\mathbb{P}\left( E\right) ),\left[ \mathbb{P}%
\left( E\right) \right] \right\rangle &=&3\sigma (N)+\left\langle
c_{1}(E)^{2}-4c_{2}(E),\left[ N\right] \right\rangle \\
\left\langle \pi ^{\ast }y_{i}\cup p_{1}(T\mathbb{P}\left( E\right) ),\left[ 
\mathbb{P}\left( E\right) \right] \right\rangle &=&0
\end{eqnarray*}%
where $\sigma (N)$ is the signature of the 4--manifold $N$\cite[SIGNATURE
THEOREM 19.4]{MiS}.
\end{example}

\subsection{Topology of conifold transitions of $B^{4}(l)\times \mathbb{C}%
P^{1}$ along $f[S^{3}]$}

It is easy to see that the definition of conifold transitions can extend to
symplectic manifolds with boundaries. In this subsection we will prove Lemma %
\ref{ml}, determining the topology of $Y_{k}$, $k=1,2$, which denote the two
conifold transitions of $B^{4}(l)\times \mathbb{C}P^{1}$ along the
Lagrangian embedding $f:S^{3}\rightarrow B^{4}(l)\times \mathbb{C}P^{1}$ in (%
\ref{mapf}).

As in Section 1, denote $\overline{%
%TCIMACRO{\U{2102} }%
%BeginExpansion
\mathbb{C}
%EndExpansion
P^{2}}$ and $S^{4}$ by $N_{k}$, $k=1,2$, respectively. Let $\sigma _{k}\in
H_{2}(N_{k})$ and $\sigma _{k}^{\ast }\in H^{2}(N_{k})$ such that $\sigma
_{1}^{\ast }$ is the dual class of the preferred generator $\sigma _{1}$ and 
$\sigma _{2}=0$, $\sigma _{2}^{\ast }=0$. As $\partial Y_{k}=\partial
B^{4}(l)\times \mathbb{C}P^{1}$, the lemma can be stated as following.

\begin{lemma}
\label{ml}Let $id_{\partial }$ denote the identity map of $\partial
Y_{k}=\partial B^{4}(l)\times \mathbb{C}P^{1}$. Then there are two
diffeomorphisms%
\begin{eqnarray*}
\phi _{1} &:&B^{4}(l)\times \mathbb{C}P^{1}\cup _{id_{\partial
}}Y_{1}\rightarrow \mathbb{P}(E_{1}^{\prime }); \\
\phi _{2} &:&B^{4}(l)\times \mathbb{C}P^{1}\cup _{id_{\partial
}}Y_{2}\rightarrow \mathbb{P}(E_{2}^{\prime })\sharp \overline{%
%TCIMACRO{\U{2102} }%
%BeginExpansion
\mathbb{C}
%EndExpansion
P^{3}}
\end{eqnarray*}%
such that the restriction of $\phi _{k}$ on $B^{4}(l)\times \mathbb{C}P^{1}$
can induce a local trivialization of the bundle $\mathbb{P}(E_{k}^{\prime })$
for $k=1,2$, where $E_{k}^{\prime }$ is the rank two complex bundle over $%
N_{k}$ with $c_{1}(E_{k}^{\prime })=-\sigma _{k}^{\ast }$ and $%
c_{2}(E_{k}^{\prime })=-1$.
\end{lemma}

To show this lemma, it needs to compute the topological invariants of $%
M_{k}:=B^{4}(l)\times \mathbb{C}P^{1}\cup _{id_{\partial }}Y_{k}$. As Smith
and Thomas \cite[Proposition 4.2]{ST} have computed the intersection forms
of the conifold transitions of $\mathbb{C}P^{2}\times \mathbb{C}P^{1}$ along
a canonical Lagrangian 3--sphere, we will extend their computation to the
topological invariants of $M_{k}$ in Lemma \ref{ml2} and Example \ref{ExM}.

The following Lemma will be very useful for the computation of invariants of 
$M_{k}$. Referring to the definition of conifold transitions recalled in
Section 2, as we have inclusions of the exceptional set $\mathbb{C}%
P^{1}\subset W_{k}^{\epsilon }$ and the set $O\times \mathbb{C}P^{1}\subset
B^{4}(l)\times \mathbb{C}P^{1}\backslash f[S^{3}]$ with $O\in B^{4}(l)$ the
original point, let $C_{k}$ and $P_{k}$ denote the images of the exceptional
set $\mathbb{C}P^{1}$ and the set $O\times \mathbb{C}P^{1}$ under the
natural inclusions $W_{k}^{\epsilon }\rightarrow Y_{k}\rightarrow M_{k}$ and 
$B^{4}(l)\times \mathbb{C}P^{1}\backslash f[S^{3}]\rightarrow
Y_{k}\rightarrow M_{k}$, respectively.

\begin{lemma}
\label{ml2}Let $\sigma \in H_{2}(%
%TCIMACRO{\U{2102} }%
%BeginExpansion
\mathbb{C}
%EndExpansion
P^{2})$ be the preferred generator. Then there are two differentiable
embeddings $r_{k}:%
%TCIMACRO{\U{2102} }%
%BeginExpansion
\mathbb{C}
%EndExpansion
P^{2}\sharp N_{k}\rightarrow M_{k}$, $k=1,2$ satisfying the following
conditions:

(i)Under the homomorphism%
\begin{equation*}
r_{k\ast }:H_{2}(%
%TCIMACRO{\U{2102} }%
%BeginExpansion
\mathbb{C}
%EndExpansion
P^{2}\sharp N_{k})\cong H_{2}(%
%TCIMACRO{\U{2102} }%
%BeginExpansion
\mathbb{C}
%EndExpansion
P^{2})\oplus H_{2}(N_{k})\rightarrow H_{2}(M_{k}),
\end{equation*}
the images of $\sigma $ and $\sigma _{k}$ are the homology classes $[P_{k}]$
and $\frac{1-(-1)^{k}}{2}\cdot \lbrack C_{k}]$ in $H_{2}(M_{k})$,
respectively.

(ii)The Euler class of the normal bundle of $r_{k}$ is%
\begin{equation*}
(-\sigma ^{\ast },-\sigma _{k}^{\ast })\in H^{2}(%
%TCIMACRO{\U{2102} }%
%BeginExpansion
\mathbb{C}
%EndExpansion
P^{2}\sharp N_{k})\cong H^{2}(%
%TCIMACRO{\U{2102} }%
%BeginExpansion
\mathbb{C}
%EndExpansion
P^{2})\oplus H^{2}(N_{k}),
\end{equation*}
where $\sigma ^{\ast }\in H^{2}(%
%TCIMACRO{\U{2102} }%
%BeginExpansion
\mathbb{C}
%EndExpansion
P^{2})$ is the dual cohomology class of $\sigma $;

(iii)In $M_{k}$, the intersection number of the submanifolds $r_{k}[%
%TCIMACRO{\U{2102} }%
%BeginExpansion
\mathbb{C}
%EndExpansion
P^{2}\sharp N_{k}]$ and $C_{k}$ is $(-1)^{k}$.
\end{lemma}

To show this lemma, first recall some results in the proof of \cite[Theorem
2.9]{STY} and \cite[Theorem 3.33]{McS}. Let%
\begin{equation*}
\Delta _{\epsilon }=\left\{ \left( u,v\right) \in T_{\epsilon }^{\ast
}S^{3}|(v_{1},v_{2},v_{3},v_{4})=\lambda (-u_{2},u_{1},-u_{4},u_{3});\lambda
\geq 0\right\}
\end{equation*}%
and fix $W_{k}$, $k=1,2$ as $W^{\pm }$ in \cite{STY}, respectively. \cite[%
Theorem 2.9]{STY} finds 4--dimensional submanifolds $\widehat{S}_{k}\subset
W_{k}^{\epsilon }$, $k=1,2$ such that

\begin{quote}
(1)$\widehat{S}_{1}$ is the complex line bundle over the exceptional set $%
\mathbb{C}P^{1}\subset W_{1}^{\epsilon }$ with Euler class $-1$ and $\psi
_{1}^{-1}[\Delta _{\epsilon }\backslash \left\{ v=0\right\} ]=\widehat{S}%
_{1}\backslash \mathbb{C}P^{1}$;\newline
(2)$\widehat{S}_{2}$ is diffeomorphic to $%
%TCIMACRO{\U{211d} }%
%BeginExpansion
\mathbb{R}
%EndExpansion
^{4}$ and $\psi _{2}^{-1}[\Delta _{\epsilon }\backslash \left\{ v=0\right\}
] $ is equal to $\widehat{S}_{2}$ with a point removed.\newline
(3)The intersection number of $\widehat{S}_{k}$ and the exceptional set $%
\mathbb{C}P^{1}$ in $W_{k}^{\epsilon }$ is $(-1)^{k}$.
\end{quote}

\noindent Considering the symplectic form $d(vdu)$ on $T^{\ast }S^{3}$ and
applying \cite[Theorem 3.33]{McS} to the Lagrangian embedding $f$, this
defines an embedding $\overline{f}:T_{\epsilon }^{\ast }S^{3}\rightarrow
B^{4}(l)\times \mathbb{C}P^{1}$ by $\overline{f}(u,v)=\exp
_{f(u)}(-J_{f(u)}\circ df_{u}\circ \Phi _{u}(v))$, where $J$ is a compatible
almost complex structure on $(B^{4}(l)\times \mathbb{C}P^{1},\omega
_{0}\times \omega _{\mathbb{C}P^{1}})$ and $\Phi _{u}:T_{u}^{\ast
}S^{3}\rightarrow T_{u}S^{3}$ is an isomorphism determined by the relation $%
\omega _{0}\times \omega _{\mathbb{C}P^{1}}(df_{u}\circ \Phi
_{u}(v),J_{f(u)}\circ df_{u}(v^{\prime }))=v(v^{\prime })$ for $v^{\prime
}\in T_{q}S^{3}$.\medskip

\begin{proof}[Proof of Lemma \protect\ref{ml2}]
As \cite[Theorem 3.33]{McS} shows that $\overline{f}$ is isotopic to a
symplectic embedding which represents a Lagrangian neighborhood of $f$, thus 
$Y_{k}$ is diffeomorphic to $B^{4}(l)\times \mathbb{C}P^{1}\backslash
f[S^{3}]\cup _{\overline{f}\circ \psi _{k}}W_{k}^{\epsilon }$. We claim that
the restriction of $\overline{f}$ on $\Delta _{\epsilon }\backslash \left\{
v=0\right\} $ is a diffeomorphism onto the relative complement of a closed
neighborhood of 
\begin{equation*}
O\times \mathbb{C}P^{1}\subset R_{0}=\left\{ \left( \overline{w},\left[ w%
\right] \right) \in B^{4}(l)\times \mathbb{C}P^{1}|\left\vert w\right\vert
<1\right\} .
\end{equation*}%
If it is true, then combining this claim with the conditions (1), (2), (3)
above and the fact that $R_{0}$ is the open disc bundle over $O\times 
\mathbb{C}P^{1}$ with Euler class $1$, it would imply that $R_{0}\cup _{%
\overline{f}\circ \psi _{k}|\psi _{k}^{-1}[\Delta _{\epsilon }\backslash
\left\{ v=0\right\} ]}\widehat{S}_{k}\cong 
%TCIMACRO{\U{2102} }%
%BeginExpansion
\mathbb{C}
%EndExpansion
P^{2}\sharp N_{k}$ are well--defined differentiable submanifolds of $%
B^{4}(l)\times \mathbb{C}P^{1}\backslash f[S^{3}]\cup _{\overline{f}\circ
\psi _{k}}W_{k}^{\epsilon }\cong Y_{k}\subset M_{k}$ for $k=1,2$,
respectively, which gives embeddings $r_{k}:%
%TCIMACRO{\U{2102} }%
%BeginExpansion
\mathbb{C}
%EndExpansion
P^{2}\sharp N_{k}\hookrightarrow M_{k}$ satisfying (i)(iii). (ii) would
follow from the fact that the restriction of the normal bundle of $%
R_{0}\subset B^{4}(l)\times \mathbb{C}P^{1}$ to $O\times \mathbb{C}P^{1}$
has Euler class $-1$ and so does the restriction of the normal bundle of $%
\widehat{S}_{1}\subset W_{1}^{\epsilon }$ to the exceptional set $\mathbb{C}%
P^{1}$.

Now it remains to show our claim. Under the identifications%
\begin{eqnarray*}
TS^{3} &=&T^{\ast }S^{3}=\left\{ \left( u,v\right) \in 
%TCIMACRO{\U{211d} }%
%BeginExpansion
\mathbb{R}
%EndExpansion
^{4}\times 
%TCIMACRO{\U{211d} }%
%BeginExpansion
\mathbb{R}
%EndExpansion
^{4}|\left\vert u\right\vert =1,\left\langle u,v\right\rangle =0\right\} , \\
%TCIMACRO{\U{211d} }%
%BeginExpansion
\mathbb{R}
%EndExpansion
^{4} &=&%
%TCIMACRO{\U{2102} }%
%BeginExpansion
\mathbb{C}
%EndExpansion
^{2}:\left( r_{1},r_{2},r_{3},r_{4}\right) \longmapsto \left(
r_{1}+ir_{2},r_{3}+ir_{4}\right) ,
\end{eqnarray*}%
it is easy to see that $v(v^{\prime })=\omega _{0}(v,Jv^{\prime })=$ $\omega
_{0}(\overline{v},J\overline{v^{\prime }})$ with $(\overline{v},\overline{%
v^{\prime }})$ the complex conjugate of $\left( v,v^{\prime }\right) \in
T_{u}^{\ast }S^{3}\times T_{u}S^{3}$. Thus for any $\left( u,v\right) \in
\Delta _{\epsilon }\backslash \left\{ v=0\right\} $, we have 
\begin{eqnarray*}
v &=&\lambda iu=\lambda \sqrt{-1}u,\lambda >0; \\
df_{u}(v) &=&\left( \overline{v},\left[ v\right] \right) =\left( \overline{v}%
,\left[ 0\right] \right) \in T_{(\overline{u},[u])}B^{4}(l)\times \mathbb{C}%
P^{1}=%
%TCIMACRO{\U{2102} }%
%BeginExpansion
\mathbb{C}
%EndExpansion
^{2}\times 
%TCIMACRO{\U{2102} }%
%BeginExpansion
\mathbb{C}
%EndExpansion
^{2}/%
%TCIMACRO{\U{2102} }%
%BeginExpansion
\mathbb{C}
%EndExpansion
u
\end{eqnarray*}
and hence $\Phi _{u}(v)=v$. These relations, together with the fact that the
complex structure $J_{f(u)}$ on $T_{f(u)}B^{4}(l)\times \mathbb{C}P^{1}$ is
induced by the multiplication by $i=\sqrt{-1}$ on $%
%TCIMACRO{\U{2102} }%
%BeginExpansion
\mathbb{C}
%EndExpansion
^{2}\times 
%TCIMACRO{\U{2102} }%
%BeginExpansion
\mathbb{C}
%EndExpansion
^{2}/%
%TCIMACRO{\U{2102} }%
%BeginExpansion
\mathbb{C}
%EndExpansion
u$, imply%
\begin{equation*}
\overline{f}(u,v)=\exp _{(\overline{u},[u])}(-\lambda \overline{u}%
,[0])=((1-\lambda )\overline{u},[u])\in R_{0}\backslash O\times \mathbb{C}%
P^{1}\text{.}
\end{equation*}%
This completes the proof.
\end{proof}

Using Lemma \ref{ml2}, we can compute the topological invariants of $%
M_{k}=B^{4}(l)\times \mathbb{C}P^{1}\cup _{id_{\partial }}Y_{k}$ for $k=1,2$.

\begin{example}
\label{ExM}We first claim that $M_{k}$ is a simply--connected 6--manifold
with $b_{3}=0$ and $H^{2}(M_{k})$ has a basis consists of $z_{k}$ and $x_{k}$%
, where $z_{k}$ is the Poincar\'{e} dual of the submanifold $R_{k}=r_{k}[%
%TCIMACRO{\U{2102} }%
%BeginExpansion
\mathbb{C}
%EndExpansion
P^{2}\sharp N_{k}]\subset M_{k}$ and the definition of $x_{k}$ is contained
in the following proof of the claim. Since $M_{k}$ is obtained by surgery
along an embedding $S^{3}\times D^{3}\hookrightarrow S^{2}\times S^{4}$ with 
$C_{k}$ the resulting 2--sphere\cite{C}, then $M_{k}$ is simply--connected
and there is a cobordism $W_{k}$ between $S^{2}\times S^{4}$ and $M_{k}$,
assuming $j_{k}:S^{2}\times S^{4}\hookrightarrow W_{k}$ and $j_{k}^{\prime
}:M_{k}\hookrightarrow W_{k}$ are the inclusions. From the cohomology exact
sequence of the pairs $(W_{k},M_{k})$ and $(W_{k},S^{2}\times S^{4})$, it is
easy to show that $H_{3}(M_{k})\cong H^{3}(M_{k})$ is trivial. Furthermore,
consider the exact sequence 
\begin{equation}
0\rightarrow H^{2}(W_{k})\overset{j_{k}^{\prime \ast }}{\rightarrow }%
H^{2}(M_{k})\overset{\delta }{\rightarrow }H^{3}(W_{k},M_{k}),  \tag{3.1}
\label{es}
\end{equation}
then $\delta z_{k}$ is a generator of $H^{3}(W_{k},M_{k})$ because the value
of $\delta z_{k}$ on the generator of $H_{3}(W_{k},M_{k})$ is equal to $%
\left\langle z_{k},[C_{k}]\right\rangle =(-1)^{k}$ by Lemma \ref{ml2} (iii).
This, together with the isomorphism $H^{2}(W_{k})\overset{j_{k}^{\ast }}{%
\rightarrow }H^{2}(S^{2}\times S^{4})$ and the exact sequence (\ref{es}),
implies that $x_{k}:=j_{k}^{\prime \ast }j_{k}^{\ast -1}a$ and $z_{k}$ form
a basis of $H^{2}(M_{k})$, where $a\in H^{2}(S^{2}\times S^{4})$ is the dual
class of the preferred generator $\left[ S^{2}\right] $ of $%
H_{2}(S^{2}\times S^{4})$. \newline
(i)The cohomology ring of $M_{k}$: The relations $j_{k\ast }^{\prime }\left[
P_{k}\right] =j_{k\ast }[S^{2}]$ and $\delta x_{k}=0$, together with Lemma %
\ref{ml2} (i) and the fact that $\left\langle x_{k},[C_{k}]\right\rangle $
is equal to the value of $\delta x_{k}\in H^{3}(W_{k},M_{k})$ on the
generator of $H_{3}(W_{k},M_{k})$, imply 
\begin{equation}
\left\langle r_{k}^{\ast }x_{k},\sigma \right\rangle =\left\langle x_{k}, 
\left[ P_{k}\right] \right\rangle =\left\langle a,[S^{2}]\right\rangle
=1;\left\langle r_{k}^{\ast }x_{k},\sigma _{k}\right\rangle =0  \tag{3.2}
\label{value}
\end{equation}%
for the basis $\sigma ,\sigma _{k}\in H_{2}(%
%TCIMACRO{\U{2102} }%
%BeginExpansion
\mathbb{C}
%EndExpansion
P^{2}\sharp N_{k})\cong H_{2}(%
%TCIMACRO{\U{2102} }%
%BeginExpansion
\mathbb{C}
%EndExpansion
P^{2})\oplus H_{2}(N_{k})$. Let $e(\nu r_{k})$ denote the Euler class of the
normal bundle $\nu r_{k}$ of $r_{k}$, then it follows from the values(\ref%
{value}) and Lemma \ref{ml2} (ii) that 
\begin{eqnarray*}
\left\langle z_{k}^{3},[M_{k}]\right\rangle &=&\left\langle
z_{k}^{2},[R_{k}]\right\rangle =\left\langle e(\nu r_{k})^{2},[%
%TCIMACRO{\U{2102} }%
%BeginExpansion
\mathbb{C}
%EndExpansion
P^{2}\sharp N_{k}]\right\rangle =\frac{1+(-1)^{k}}{2}; \\
\left\langle z_{k}x_{k}^{2},[M_{k}]\right\rangle &=&\left\langle
x_{k}^{2},[R_{k}]\right\rangle =\left\langle r_{k}^{\ast }x_{k}^{2},[%
%TCIMACRO{\U{2102} }%
%BeginExpansion
\mathbb{C}
%EndExpansion
P^{2}\sharp N_{k}]\right\rangle =1; \\
\left\langle x_{k}z_{k}^{2},[M_{k}]\right\rangle &=&\left\langle
x_{k}z_{k},[R_{k}]\right\rangle =\left\langle r_{k}^{\ast }x_{k}\cup e(\nu
r_{k}),[%
%TCIMACRO{\U{2102} }%
%BeginExpansion
\mathbb{C}
%EndExpansion
P^{2}\sharp N_{k}]\right\rangle =-1; \\
\left\langle x_{k}^{3},[M_{k}]\right\rangle &=&\left\langle j_{k}^{^{\prime
}\ast }j_{k}^{\ast -1}a^{3},[M_{k}]\right\rangle =0\text{.}
\end{eqnarray*}%
(ii)The first Pontrjagin class of $M_{k}$: The exact sequence%
\begin{equation*}
H_{7}(W_{k})\overset{\partial }{\rightarrow }H_{6}(S^{2}\times S^{4}\sqcup
M_{k})\rightarrow H_{6}(W_{k}),
\end{equation*}
together with the relations $\partial \left[ W_{k}\right] =[M_{k}]-[S^{2}%
\times S^{4}]$, $p_{1}(M_{k})=j_{k}^{\prime \ast }p_{1}(W_{k})$ and $%
\left\langle p_{1}(W_{k})\cup j_{k}^{\prime \ast -1}x_{k},j_{k\ast
}[S^{2}\times S^{4}]\right\rangle =\left\langle p_{1}(S^{2}\times S^{4})\cup
a,[S^{2}\times S^{4}]\right\rangle =0$, imply 
\begin{equation*}
\left\langle p_{1}(M_{k})x_{k},[M_{k}]\right\rangle =\left\langle
p_{1}(W_{k})\cup j_{k}^{\prime \ast -1}x_{k},j_{k\ast }^{\prime
}[M_{k}]-j_{k\ast }[S^{2}\times S^{4}]\right\rangle =0\text{.}
\end{equation*}%
From the relations $p_{1}(\nu r_{k})=e(\nu r_{k})^{2}$, $\left\langle p_{1}(%
%TCIMACRO{\U{2102} }%
%BeginExpansion
\mathbb{C}
%EndExpansion
P^{2}\sharp N_{k}),\left[ 
%TCIMACRO{\U{2102} }%
%BeginExpansion
\mathbb{C}
%EndExpansion
P^{2}\sharp N_{k}\right] \right\rangle =3\cdot \frac{1+(-1)^{k}}{2}$ and $%
z_{k}\cap \lbrack M_{k}]=r_{k\ast }[%
%TCIMACRO{\U{2102} }%
%BeginExpansion
\mathbb{C}
%EndExpansion
P^{2}\sharp N_{k}]$, together with Lemma \ref{ml2} (ii) and the
decomposition $r_{k}^{\ast }TM_{k}=T(%
%TCIMACRO{\U{2102} }%
%BeginExpansion
\mathbb{C}
%EndExpansion
P^{2}\sharp N_{k})\oplus \nu r_{k}$, we get 
\begin{equation*}
\left\langle p_{1}(M_{k})z_{k},[M_{k}]\right\rangle =\left\langle
r_{k}^{\ast }p_{1}(M_{k}),[%
%TCIMACRO{\U{2102} }%
%BeginExpansion
\mathbb{C}
%EndExpansion
P^{2}\sharp N_{k}]\right\rangle =2\times (1+(-1)^{k}).
\end{equation*}%
(iii)The second Whitney class of $M_{k}$: As the value $w_{2}(S^{2}\times
S^{4})=0$ and the isomorphism $j_{k}^{\ast }:H^{2}(W_{k})\rightarrow
H^{2}(S^{2}\times S^{4})$ imply $w_{2}(W_{k})=0$, thus%
\begin{equation*}
w_{2}(M_{k})=j_{k}^{\prime \ast }w_{2}(W_{k})=0.
\end{equation*}
\end{example}

Now we can prove the Lemma \ref{ml}.

\begin{proof}[Proof of Lemma \protect\ref{ml}]
Denote $S^{6}$ and $\overline{%
%TCIMACRO{\U{2102} }%
%BeginExpansion
\mathbb{C}
%EndExpansion
P^{3}}$ by $Q_{1}$ and $Q_{2}$, respectively. By Wall and Jupp's
classification of simply--connected 6--manifolds with torsion--free homology%
\cite{W}\cite{J}, comparing the invariants of $M_{k}$ and $\mathbb{P}%
(E_{k}^{\prime })$ (see Example\ref{ExM} and Example\ref{ExP}), we get two
diffeomorphisms 
\begin{equation*}
\varphi _{k}:M_{i}\rightarrow \mathbb{P}(E_{k}^{\prime })\sharp Q_{k}\text{, 
}k=1,2
\end{equation*}%
such that $\varphi _{k}^{\ast }a_{k}=x_{k}+\frac{1+(-1)^{k}}{2}\cdot z_{k}$
for $k=1,2$, $\varphi _{1}^{\ast }\pi _{1}^{\ast }(-\sigma _{1}^{\ast
})=-z_{1}$ and $\varphi _{2}^{\ast }z^{\prime }=z_{2}$, where 
\begin{equation*}
a_{k}\in H^{2}(\mathbb{P}(E_{k}^{\prime })\sharp Q_{k})\cong H^{2}(\mathbb{P}%
(E_{k}^{\prime }))\oplus H^{2}(Q_{k})
\end{equation*}%
denote the first Chern classes of the dual bundles of the tautological line
bundles over $\mathbb{P}(E_{k}^{\prime })$ for $k=1,2$, respectively, $\pi
_{1}:P(E_{1}^{\prime })\rightarrow \overline{%
%TCIMACRO{\U{2102} }%
%BeginExpansion
\mathbb{C}
%EndExpansion
P^{2}}$ is the bundle projection, and%
\begin{equation*}
z^{\prime }\in H^{2}(\mathbb{P}(E_{2}^{\prime })\sharp \overline{%
%TCIMACRO{\U{2102} }%
%BeginExpansion
\mathbb{C}
%EndExpansion
P^{3}})\cong H^{2}(\mathbb{P}(E_{2}^{\prime }))\oplus H^{2}(\overline{%
%TCIMACRO{\U{2102} }%
%BeginExpansion
\mathbb{C}
%EndExpansion
P^{3}})
\end{equation*}%
is the Poincar\'{e} dual of the submanifold $%
%TCIMACRO{\U{2102} }%
%BeginExpansion
\mathbb{C}
%EndExpansion
P^{2}\subset \overline{%
%TCIMACRO{\U{2102} }%
%BeginExpansion
\mathbb{C}
%EndExpansion
P^{3}}$.

We claim that $\varphi _{k\ast }[O\times \mathbb{C}P^{1}]=f_{k\ast }[\mathbb{%
C}P^{1}]$ for the submanifold $O\times \mathbb{C}P^{1}\subset B^{4}(l)\times 
\mathbb{C}P^{1}\subset M_{k}$ and embeddings $f_{k}:\mathbb{C}%
P^{1}\rightarrow \mathbb{P}(E_{k}^{\prime })\sharp Q_{k}$ representing a
fiber of $\mathbb{P}(E_{k}^{\prime })$. As the relations $\left\langle
z_{k},[O\times \mathbb{C}P^{1}]\right\rangle =0$ and $j_{k\ast }^{\prime }%
\left[ P_{k}\right] =j_{k\ast }[S^{2}]=j_{k\ast }^{\prime }[O\times \mathbb{C%
}P^{1}]$ imply that $[O\times \mathbb{C}P^{1}]$ is the dual base of $x_{k}+%
\frac{1+(-1)^{k}}{2}\cdot z_{k}=\varphi _{k}^{\ast }a_{k}$ in the basis 
\begin{equation*}
\left\{ x_{k}+\frac{1+(-1)^{k}}{2}\cdot z_{k},z_{k}\right\} =\left\{ 
\begin{array}{c}
\{\varphi _{1}^{\ast }a_{1},\varphi _{1}^{\ast }\pi _{1}^{\ast }(-\sigma
_{1}^{\ast })\}\text{ for }k=1, \\ 
\{\varphi _{2}^{\ast }a_{2},\varphi _{2}^{\ast }z^{\prime }\}\text{ for }k=2,%
\end{array}%
\right. 
\end{equation*}%
comparing this with the fact that $f_{k\ast }[\mathbb{C}P^{1}]$ is the dual
base of $a_{k}$ in the basis $\left\{ a_{1},\pi _{1}^{\ast }(-\sigma
_{1}^{\ast })\right\} $ for $k=1$ and in the basis $\left\{ a_{2},z^{\prime
}\right\} $ for $k=2$, respectively, shows the claim.

Since the claim above implies that $\varphi _{k}|_{O\times \mathbb{C}P^{1}}$
is homotopic to $f_{k}$, then by \cite[THEOREM 1]{Ha} and the isotopy
extension theorem\cite[Chapter 8, 1.3. Theorem]{Hi}, there is an isotopy $%
F_{t}^{k}:\mathbb{P}(E_{k}^{\prime })\sharp Q_{k}\rightarrow \mathbb{P}%
(E_{k}^{\prime })\sharp Q_{k}$, $0\leq t\leq 1$, such that $F_{0}^{k}=id$
and $F_{1}^{k}\circ \varphi _{k}|_{O\times \mathbb{C}P^{1}}=f_{k}$. Let $%
\overline{f_{k}}:B^{4}(l)\times \mathbb{C}P^{1}\rightarrow \mathbb{P}%
(E_{k}^{\prime })\sharp Q_{k}$ be an extension of $f_{k}$ which can induce a
local trivialization of the bundle $\mathbb{P}(E_{k}^{\prime })$, then $%
F_{1}^{k}\circ \varphi _{k}|_{B^{4}(l)\times \mathbb{C}P^{1}}$ and $%
\overline{f_{k}}$ determine two closed tubular neighborhoods of $f_{k}[%
\mathbb{C}P^{1}]$. By the ambient isotopy theorem for closed tubular
neighborhoods\cite[Chapter 4, Section 6, Exercises 9]{Hi}, there exists an
isotopy $H_{t}^{k}:\mathbb{P}(E_{k}^{\prime })\sharp Q_{k}\rightarrow 
\mathbb{P}(E_{k}^{\prime })\sharp Q_{k}$, $0\leq t\leq 1$, such that $%
H_{0}^{k}=id$, $H_{1}^{k}\circ F_{1}^{k}\circ \varphi _{k}[B^{4}(l)\times 
\mathbb{C}P^{1}]=\overline{f_{k}}[B^{4}(l)\times \mathbb{C}P^{1}]$ and%
\begin{equation*}
g:=\overline{f_{k}}^{-1}\circ H_{1}^{k}\circ F_{1}^{k}\circ \varphi
_{k}|_{B^{4}(l)\times \mathbb{C}P^{1}}:B^{4}(l)\times \mathbb{C}%
P^{1}\rightarrow B^{4}(l)\times \mathbb{C}P^{1}
\end{equation*}%
is a $B^{4}(l)-$bundle isomorphism. As the homotopy group $\pi _{2}(O(4))$
of the real orthogonal group $O(4)$ is trivial, this implies $g|_{\partial
B^{4}(l)\times \mathbb{C}P^{1}}$ is isotopic to the identity map of $%
\partial B^{4}(l)\times \mathbb{C}P^{1}$ and then similar to the proof of 
\cite[Chapter 8, 2.3]{Hi}, we can extend $g$ to a self--diffeomorphism $\phi 
$ of $M_{k}=B^{4}(l)\times \mathbb{C}P^{1}\cup _{id_{\partial }}Y_{k}$ which
is identity outside a neighborhood of $B^{4}(l)\times \mathbb{C}P^{1}$.
Consequently, the restriction of $\phi _{k}:=H_{1}^{k}\circ F_{1}^{k}\circ
\varphi _{k}\circ \phi ^{-1}$ on $B^{4}(l)\times \mathbb{C}P^{1}$ is equal
to $\overline{f_{k}}$ and hence $\phi _{k},k=1,2,$ are the desired
diffeomorphisms.
\end{proof}

\subsection{Topology of symplectic conifold transitions of $\mathbb{C}P^{1}-$%
bundles}

The establishment of Lemma \ref{ml} make it possible to prove Theorem \ref%
{mt}, which determines the diffeomorphism types of symplectic conifold
transitions of $\mathbb{C}P^{1}-$bundles over 4--manifolds along canonical
Lagrangian 3--spheres. In this section, we show this theorem and Corollary %
\ref{mc}.

\begin{proof}[Proof of the theorem \protect\ref{mt}]
From \cite[Theorem 2.9]{STY} and the definition of  the two symplectic
conifold transitions $Z_{k}$, $k=1,2$ along a canonical Lagrangian embedding 
$S^{3}\overset{f}{\rightarrow }B^{4}(l)\times \mathbb{C}P^{1}\overset{\eta }{%
\rightarrow }\mathbb{P}\left( E\right) $, we get the identification%
\begin{equation*}
Z_{k}=\mathbb{P}\left( E\right) \cup _{\eta }M_{k}\diagdown (\text{Interior }%
\eta \lbrack B^{4}(l)\times \mathbb{C}P^{1}])
\end{equation*}%
as almost complex manifolds with $B^{4}(l)\times \mathbb{C}P^{1}$ seen as a
subset of $M_{k}=B^{4}(l)\times \mathbb{C}P^{1}\cup _{id_{\partial }}Y_{k}$.
Denote $S^{6}$ and $\overline{%
%TCIMACRO{\U{2102} }%
%BeginExpansion
\mathbb{C}
%EndExpansion
P^{3}}$ by $Q_{1}$ and $Q_{2}$, respectively, and let $E\cup _{%
%TCIMACRO{\U{2102} }%
%BeginExpansion
\mathbb{C}
%EndExpansion
^{2}}E_{k}^{\prime }$ denote the complex vector bundle over the one point
union $N\vee N_{k}$ obtained by identifying one fiber $%
%TCIMACRO{\U{2102} }%
%BeginExpansion
\mathbb{C}
%EndExpansion
^{2}$ of the two bundles $E$ and $E_{k}^{\prime }$, respectively. The
identity map $id$ of $\mathbb{P}\left( E\right) $ and the diffeomorphisms $%
\phi _{k}:M_{k}\rightarrow \mathbb{P}\left( E_{k}^{\prime }\right) \sharp
Q_{k}$ in Lemma \ref{ml} contribute to define diffeomorphisms%
\begin{equation*}
\Psi _{k}:Z_{k}\overset{\cong }{\rightarrow }\mathbb{P}\left( E_{k}\right)
\sharp Q_{k},k=1,2
\end{equation*}%
where $E_{k}$ is the pullback bundle of the bundle $E\cup _{%
%TCIMACRO{\U{2102} }%
%BeginExpansion
\mathbb{C}
%EndExpansion
^{2}}E_{k}^{\prime }$ under the natural map $N\sharp N_{k}\rightarrow N\vee
N_{k}$. It is very easy to get the Chern class of $E_{k}$ from the
isomorphism $H^{2}(N\vee N_{k})\cong H^{2}(N\sharp N_{k})$, the homomorphism%
\begin{equation*}
H^{4}(N\vee N_{k})\cong 
%TCIMACRO{\U{2124} }%
%BeginExpansion
\mathbb{Z}
%EndExpansion
\oplus 
%TCIMACRO{\U{2124} }%
%BeginExpansion
\mathbb{Z}
%EndExpansion
\rightarrow H^{4}(N\sharp N_{k})\cong 
%TCIMACRO{\U{2124} }%
%BeginExpansion
\mathbb{Z}
%EndExpansion
:(a,b)\longmapsto a+b
\end{equation*}
and the values%
\begin{equation*}
c_{j}(E\cup _{%
%TCIMACRO{\U{2102} }%
%BeginExpansion
\mathbb{C}
%EndExpansion
^{2}}E_{k}^{\prime })=(c_{j}(E),c_{j}(E_{k}))\in H^{2j}(N)\oplus
H^{2j}(N_{k})\cong H^{2j}(N\vee N_{k})
\end{equation*}%
for $j=1,2$. 

Assume $N$ is symplectic. To prove the diffeomorphisms $\Psi _{k}$ preserve $%
c_{1}$, consider the commutative diagram%
\begin{equation}
\begin{array}{ccc}
H^{2}(\mathbb{P}\left( E_{k}\right) \sharp Q_{k}) & \overset{\Psi _{k}^{\ast
}}{\rightarrow } & H^{2}(Z_{k}) \\ 
\uparrow \cong  &  & \uparrow \cong  \\ 
H^{2}(\mathbb{P}\left( E\right) \cup _{\eta \circ \overline{f_{k}}^{-1}}%
\mathbb{P}\left( E_{k}^{\prime }\right) \sharp Q_{k}) & \overset{(id\cup
\phi _{k})^{\ast }}{\rightarrow } & H^{2}(\mathbb{P}\left( E\right) \cup
_{\eta }M_{k}) \\ 
\downarrow  &  & \downarrow  \\ 
H^{2}(\mathbb{P}\left( E\right) )\oplus H^{2}(\mathbb{P}\left( E_{k}^{\prime
}\right) \sharp Q_{k}) & \overset{id^{\ast }\oplus \phi _{k}^{\ast }}{%
\rightarrow } & H^{2}(\mathbb{P}\left( E\right) )\oplus H^{2}(M_{k})%
\end{array}
\tag{3.3}  \label{diagram}
\end{equation}%
where $\overline{f_{k}}:B^{4}(l)\times \mathbb{C}P^{1}\rightarrow \mathbb{P}%
\left( E_{k}^{\prime }\right) \sharp Q_{k}$ is the restriction of $\phi _{k}$
as in the proof of Lemma \ref{ml} and the vertical homomorphisms are induced
by the natural inclusions. As the conifold transitions is an almost complex
operation preserving the first Chern class \cite{STY}\cite[Lemma 2]{CS}, the
formula of the first Chern class of a blowup at a point \cite[p.608]{GH} and 
$c_{1}(T\mathbb{P}\left( E\right) )=2a+\pi ^{\ast }(c_{1}(TN)+c_{1}(E))$ by
Example \ref{ExP}(ii), imply that the images of $c_{1}(T\mathbb{P}\left(
E_{k}\right) \sharp Q_{k})$ and $c_{1}(TZ_{k})$ under the vertical composite
homomorphisms are%
\begin{equation}
\left( 2a+\pi ^{\ast }(c_{1}(TN)+c_{1}(E)),2a_{k}-(1+(-1)^{k})\cdot
z^{\prime }\right) ,  \tag{3.4}  \label{im_1}
\end{equation}%
\begin{equation}
\left( 2a+\pi ^{\ast }(c_{1}(TN)+c_{1}(E)),2x_{k}\right) ,  \tag{3.5}
\label{im_2}
\end{equation}%
respectively, with $a_{k}$, $z^{\prime }$ and $x_{k}$ defined in the proof
of Lemma \ref{ml} and Example \ref{ExM}. Since $\phi _{k}^{\ast }a_{k}=x_{k}+%
\frac{1+(-1)^{k}}{2}\cdot z_{k}$, $\phi _{2}^{\ast }z^{\prime }=z_{2}$ by
the proof of Lemma \ref{ml}, then the horizontal homomorphism $id^{\ast
}\oplus \phi _{k}^{\ast }$ maps the class (\ref{im_1}) to (\ref{im_2}) and
hence $c_{1}(TZ_{k})=\Psi _{k}^{\ast }c_{1}(T\mathbb{P}\left( E_{k}\right) )$
as the vertical homomorphisms in the diagram (\ref{diagram}) are injective.
This completes the proof.
\end{proof}

Now we turn to show Corollary \ref{mc}.

\begin{proof}[Proof of Corollary \protect\ref{mc}]
As the blowup of a K\"{a}hler manifold at a point is also K\"{a}hler\cite[%
Proposition 3.24]{V}, this Corollary follows easily from Theorem \ref{mt}
and the claim that both $E_{k}$ over the projective surfaces $N\sharp N_{k}$
admit holomorphic structures. To prove the claim, it suffices to note
Schwarzenberger \cite[Theorem 9]{S} showed that a complex vector bundle over
a projective surface $S$ admits a holomorphic structure if and only if the
first Chern class of the bundle belongs to $H^{1,1}(S)$. As $%
c_{1}(E_{2})=c_{1}(E)$ and $c_{1}(E_{1})$ is equal to $c_{1}(E)$ plus the
exceptional divisor $-\sigma _{1}^{\ast }$, so $c_{1}(E_{k})\in
H^{1,1}(N\sharp N_{k})$ by the Lefschetz theorem on (1,1) classes\cite[%
Theorem 11.30]{V}. This completes the proof.
\end{proof}


\begin{thebibliography}{99}
\bibitem{ALP} M. Audin, F. Lalonde and L. Polterovich, in Holomorphic curves
in symplectic geometry, 271--321, Progr. Math., 117, Birkha\"{u}ser, Basel,
1994.

\bibitem{BT} R. Bott and L. W. Tu, Differential forms in algebraic topology,
Springer, New York, 1982.

\bibitem{C} H. C. Clemens, Double solids, Adv. in Math. 47 (1983), no. 2,
107--230.

\bibitem{CS} A. Corti and I. Smith, Conifold transitions and Mori theory,
Math. Res. Lett. 12 (2005), no. 5--6, 767--778.

\bibitem{GH} P. Griffiths and J. Harris, Principles of algebraic geometry,
reprint of the 1978 original, Wiley Classics Lib., Wiley, New York, 1994.

\bibitem{GS} R. Gompf and A. Stipsicz, 4--manifolds and Kirby calculus,
Grad. Stud. Math., 20, Amer. Math. Soc., 1999.

\bibitem{Ha} A. Haefliger, Differentiable embeddings, Bull. Amer. Math. Soc.
67 (1961), 109--112.

\bibitem{Hi} M. Hirsch, Differential Topology, corrected reprint of the 1976
original, Grad. Texts in Math., 33, Springer, New York, 1994.

\bibitem{J} P. Jupp, Classification of certain 6--manifolds, Proc. Camb.
Phil. Soc. 73 (1973), 293--300.

\bibitem{McS} D. Mcduff and D. Salamon, Introduction to symplectic topology,
Second edition, Oxford Univ. Press, New York, 1998.

\bibitem{MiS} J. Milnor and J. Stasheff, Characteristic classes, Princeton
Univ. Press, Princeton, NJ, 1974.

\bibitem{Sa} D. Salamon, Uniqueness of symplectic structures, Acta Math.
Vietnam. 38 (2013), no. 1, 123--144.

\bibitem{S} R. L. E. Schwarzenberger, Vector bundles on algebraic surface,
Proc. London Matb. Soc. (3) 11 (1961) 601--22.

\bibitem{ST} I. Smith and R.P. Thomas, Symplectic surgeries from
singularities, Turk. J. Math. 27 (2003), 231--250.

\bibitem{STY} I. Smith, R.P. Thomas and S.-T. Yau, Symplectic conifold
transitions, J. Differential Geom. 62 (2002), no. 2, 209--242.

\bibitem{V} C. Voisin, Hodge theory and complex algebraic geometry I,
translated from the French original by Leila Schneps, Cambridge Stud. Adv.
Math., 76, Cambridge Univ. Press, Cambridge, 2002.

\bibitem{W} C. T. C. Wall, Classification problems in differential topology.
V. On certain 6--manifolds, Invent. Math. 1 (1966), 335--374.
\end{thebibliography}
\end{document}